\documentclass{proc-l}
\usepackage{amssymb}

\newcommand{\cyc}[1]{\langle#1\rangle}

\newtheorem{theorem}{Theorem}[section]
\newtheorem{lemma}[theorem]{Lemma}

\input{tcilatex}

\begin{document}
\title{ON A GENERALIZATION OF BAER THEOREM}
\author{L. A. KURDACHENKO}
\address{Department of Algebra, National University of Dnepropetrovsk}
\curraddr{Vul. Naukova 13, Dnepropetrovsk 50, Ukraine 49050}
\email{lkurdachenko@gmail.com}
\thanks{The authors were supported by Proyecto MTM2010-19938-C03-03 of
MICINN (Spain), the Government of Arag\'{o}n (Spain) and FEDER funds from
European Union}
\author{J. OTAL}
\address{Department of Mathematics - IUMA, University of Zaragoza}
\curraddr{Pedro Cerbuna 12, 50009 Zaragoza, Spain}
\email{otal@unizar.es}
\thanks{}
\author{I. Ya. SUBBOTIN}
\address{Department of Mathematics and Natural Sciences, National University}
\curraddr{5245 Pacific Concourse Drive, LA, CA 90045, USA}
\email{isubboti@nu.edu}
\thanks{}
\date{}
\dedicatory{}

\begin{abstract}
R. Baer has proved that if the factor-group $G/\zeta_n(G)$ of a group $G$ by
the member $\zeta_n(G)$ of its upper central series is finite (here $n$ is a
positive integer) then the member $\gamma_{n+1}(G)$ of the lower central
series of $G$ is also finite. In particular, in this case, the nilpotent
residual of $G$ is finite. This theorem admits the following simple
generalization that has been published very recently by M. de Falco, F. de
Giovanni, C. Musella and Ya. P. Sysak: "If the factor-group $G/Z$ of a group 
$G$ modulo its upper hypercenter $Z$ is finite then G has a finite normal
subgroup $L$ such that $G/L$ is hypercentral". In the current article we
offer a new simpler very short proof of this theorem and specify it
substantially. In fact, we prove that if $|G/Z| = t$ then $|L|\leq t^k$,
where $k = \frac{1}{2}(log_pt+1)$, and $p$ is the least prime divisor of $t$.
\end{abstract}

\maketitle

\section{Introduction}

One of the important long-standing results in the Theory of Groups is a
classical theorem due to I. Schur \cite{SI1904}, which establishes a
connection between the factor-group $G/\zeta(G)$ of a group $G$ modulo its
center $\zeta(G)$ and the derived subgroup $[G,G]$ of $G$. It follows from
Schur's theorem \cite{SI1904} that \textit{if $G/\zeta(G)$ is finite then $%
[G,G]$ is also finite}. A natural question related to this result appears
here, namely the question regarding the relationship between the orders $%
|G/\zeta(G)|$ and $|[G,G]|$. J. Wiegold in the paper \cite{WJ1956} obtained
the following answer to this question. Let $G$ be a group such that $%
|G/\zeta(G)| = t$ is finite. J. Wiegold proved that there exists a function $%
w$ such that $|[G,G]|\leq w(t)$. He also was able to obtain for this
function the value $w(t) = t^m$ where $m = \frac{1}{2}(log_pt-1)$ and $p$ is
the least prime divisor of $t$. Later on, J. Wiegold was able to show that
this boundary value may be attained if and only if $t = p^n$ for some prime $%
p$ (\cite{WJ1965}). When $t$ has more than one prime divisor, the picture
becomes more complicated.

\medskip

Various generalizations of Schur's theorem can be found in the mathematical
literature. One of the most interesting approaches would be studying the
properties of the following question: \textit{study properties of the
factor-group $G/\zeta(G)$ such that the derived subgroup $[G,G]$ satisfies
the same property}. A class of groups $\mathfrak{X}$ is said to be \textit{a
Schur class} if for every group $G$ such that $G/\zeta(G)\in\mathfrak{X}$
the derived subgroup $[G,G]$ also belongs to $\mathfrak{X}$. Schur's classes
were introduced in the paper \cite{FGK1995}. Besides of the obvious examples
of the classes of finite and of locally finite groups, the class of
polycyclic--by--finite groups and the class of Chernikov groups are also
Schur's classes (see, for example, \cite[Theorem 3.9]{KOS2007}). In this
paper \cite{FGK1995} other Schur's classes were found as well.

\medskip

In the paper \cite{BR1952} R. Baer generalized Schur's theorem in a
different direction. We recall that \textit{the upper central series of a
group $G$} is the ascending series 
\begin{equation*}
\langle 1\rangle =
\zeta_0(G)\leq\zeta_1(G)\leq\cdots\leq\zeta_\alpha(G)\leq\zeta_{\alpha+1}(G)%
\leq\cdots\zeta_\delta(G) = \zeta_\infty(G) 
\end{equation*}
given by $\zeta_1(G) = \zeta(G)$ is the center of G, and recursively $%
\zeta_{\alpha+1}(G)/\zeta_\alpha(G) = \zeta(G/\zeta_\alpha(G))$ for all
ordinals $\alpha$ and $\zeta_\lambda(G) = \bigcup_{\mu<\lambda}\zeta_\mu(G)$
for every limit ordinal $\lambda$. The last term $\zeta_\infty(G)$ of this
series is called \textit{the upper hypercenter of $G$}. $G$ itself is called 
\textit{hypercentral} if $\zeta_\infty(G)=G$. In general, the length of the
upper central series of $G$ is denoted by $zl(G)$. On the other hand, 
\textit{the lower central series of $G$} is the descending series 
\begin{equation*}
G =
\gamma_1(G)\geq\gamma_2(G)\geq\cdots\geq\gamma_\alpha(G)\geq\gamma_{%
\gamma+1}(G)\geq\cdots 
\end{equation*}
given by $\gamma_2(G) = [G,G]$, and recursively $\gamma_{\alpha+1}(G) =
[\gamma_\alpha(G),G]$ for all ordinals $\alpha$ and $\gamma_\lambda(G) =
\bigcap_{\mu<\lambda}\gamma_\mu(G)$ for every limit ordinal $\lambda $.

\medskip

R. Baer proved that \textit{if for some positive integer $n$ the
factor-group $G/\zeta_n(G)$ is finite, then $\gamma_{n+1}(G)$ is finite too}
(\cite{BR1952}). In particular, in this case the nilpotent residual of $G$
(that is, the intersection of all normal subgroups $N$ of $G$ such that $G/N$
is nilpotent) is finite. Very recently, in the paper \cite{FGMS2011}, M. de
Falco, F. de Giovanni, C. Musella and Ya. P. Sysak obtained the following
generalization of this result:

\medskip

\noindent\textbf{Theorem A}. \textit{Let $G$ be a group and let $Z$ be the
upper hypercenter of $G$. If $G/Z$ is finite, then $G$ has a finite normal
subgroup $L$ such that $G/L$ is hypercentral}.

\medskip

In Section \ref{S2} we provide an elementary proof of this result, which is
considerably shorter than the original one.

\medskip

Just as in the theorem of Schur, the question on finding a relationship
between the factor-group $G/\zeta_\infty(G)$ and the hypercentral residual
of $G$ (the intersection of all normal subgroups $N$ of $G$ such that $G/N$
is hypercentral) appears to be very natural. More specifically, \textit{is
there a function (dependeding on the order of $G/\zeta_\infty(G)$) that
bounds the order of the hypercentral residual of G?}. In this note we show
that Theorem A can be significantly improved. We prove that the order of the
hypercentral residual of $G$ is bounded by a function of the order of $%
G/\zeta_\infty(G)$ and moreover we are able to give an explicit form of this
function. Thus the main result of the current note is the following

\medskip

\noindent\textbf{Theorem B}. \textit{Let $G$ be a group and let $Z$ be the
upper hypercenter of $G$. Suppose that $G/Z$ is finite and put $|G/Z| = t$.
Then $G$ has a finite normal subgroup $L$ such that $G/L$ is hypercentral.
Moreover, $|L|\leq t^k$, where $k = \frac{1}{2}(log_pt+1)$ and $p$ is the
least prime divisor of $t$}.

\section{A short proof of Theorem A}

\label{S2}

The proof makes use of an auxiliary result by N. S. Hekster \cite[Lemma 2.4]%
{HN1986}.

\begin{lemma}[HN1986]
\label{l1} Let $G$ be a group, $K$ a subgroup of $G$, and suppose that $G =
K\zeta_n(G)$ for some positive integer $n$. Then the following properties
holds.

\begin{enumerate}
\item $\gamma_{n+1}(G) = \gamma_{n+1}(K)$.

\item $\zeta_n(K) = K\cap\zeta_n(G)$.

\item $\gamma_{n+1}(G)\cap\zeta_n(G) = \gamma_{n+1}(K)\cap\zeta_n(K)$.
\end{enumerate}
\end{lemma}

\noindent\textit{Proof of Theorem A}. We note that if $zl(G)$ is finite, the
result follows from Baer's theorem \cite{BR1952}. Therefore we may suppose
that $zl(G)$ is infinite. Let $K$ be a finitely generated subgroup with the
property $G = ZK$. We have that $K$ is nilpotent--by--finite (see \cite[%
Proposition 3.19]{KOS2007} for example). Since $G/Z$ is not nilpotent,
neither is $K$. Set $r = zl(K)$ and let $C$ be the upper hypercenter of $K$.
We claim that $C=C\cap Z$. For, otherwise $CZ/Z\ne\langle 1\rangle$, which
means that the upper hypercenter of $G/Z$ is not identity, a contradiction.
Then $C = C\cap Z$ as claimed. By Baer's theorem \cite{BR1952}, $%
\gamma_{r+1}(K)$ is finite. It follows that the nilpotent residual $L$ of $K$
is finite.

We now consider the local system $\mathcal{L}$ consisting of all finitely
generated subgroups of $G$ that contains $K$. Pick $V\in\mathcal{L}$ and let 
$C_V$ be the upper hypercenter of $V$. Clearly we have $G = ZV$ and then $%
C_V = V\cap Z$. Since $V\leq KZ$ and $K\leq V$, we have $V = K(V\cap Z) =
KC_V$. Put $n = zl(V)$. Since $V = KC_V$, we have that $\gamma_{n+1}(V) =
\gamma_{n+1}(K)$ by Lemma \ref{l1}. In particular, $\gamma_{n+1}(K)$ is
normal in $V$. Since $L$ is a characteristic subgroup of $\gamma_{n+1}(K)$, $%
L$ is normal in $V$. Since this holds for each $V\in\mathcal{L}$, $L$ is
normal in $G=\bigcup_{V\in\mathcal{L}}V$. We have 
\begin{equation*}
G/ZL\cong (G/L)(ZL/L) = (KZ/L)/(ZL/L) = (K/L)(ZL/L)/(ZL/L)\cong 
\end{equation*}
\begin{equation*}
\cong (K/L)/((K/L)\cap (ZL/L)). 
\end{equation*}
Since $K/L$ is nilpotent, so is $G/ZL$. Since the hypercenter of $G/L$
includes $ZL/L$, $G/L$ has to be hypercentral.\hfill $\Box$

\section{Proof of Theorem B}

\label{S3}

Let $G$ be a group, $R$ a ring and $A$ an $RG$--module. We construct \textit{%
the upper $RG$--central series of $A$} as the ascending chain of submodules 
\begin{equation*}
\{0\} = A_0\leq A_1\leq\cdots\leq A_\alpha\leq A_{\alpha+1}\cdots A_\lambda, 
\end{equation*}
where $A_1 = \zeta_{RG}(A) = \{a\in A\ |\ a(g-1) = 0\}$, $%
A_{\alpha+1}/A_\alpha = \zeta_{RG}(A/A_\alpha)$ for every ordinal $%
\alpha<\lambda$ and $\zeta_{RG}(A/A_\lambda) = \{0\}$. The last term $%
A_\lambda$ of this series is called \textit{the upper $RG$--hypercenter of $A
$} and will denoted by $\zeta_{RG}^\infty(A)$. If $A = A_\lambda$, then $A$
is said to be \textit{$RG$--hypercentral}. Moreover, if $\lambda$ is finite,
then $A$ is said to be \textit{$RG$--nilpotent}.

Let $B\leq C$ be $RG$--submodules of $A$. The factor $C/B$ is called \textit{%
$G$--eccentric} if $C_G(C/B)\ne G$. An $RG$--submodule $C$ of $A$ is said to
be \textit{$RG$--hypereccentric} if it has an ascending series of $RG$%
--submodules 
\begin{equation*}
\{0\} = C_0\leq C_1\leq\cdots\leq C_\alpha\leq C_{\alpha+1}\leq\cdots
C_\lambda = C 
\end{equation*}
such that every factor $C_{\alpha+1}/C_\alpha$ is a $G$--eccentric simple $RG
$--module.

\medskip

It is said that \textit{the $RG$--module $A$ has the $Z$--decomposition} if
we can express 
\begin{equation*}
A = \zeta_{RG}^\infty(A)\bigoplus E_{RG}^\infty(A), 
\end{equation*}
where $E_{RG}^\infty(A)$ is the maximal $RG$--hypereccentric $RG$--submodule
of $A$ (D. I. Zaitsev \cite{ZD1979}). We note that, if $A$ has the $Z$%
--decomposition, then $E_{RG}^\infty(A)$ includes every $RG$--hypereccentric 
$RG$--submodule and, in particular, it is unique. Indeed, put $%
E=E_{RG}^\infty(A)$ and let $B$ be a $RG$--hypereccentric $RG$--submodule of 
$A$. If $(B + E)/E$ is non-zero, then it has a non-zero simple $RG$%
--submodule $U/E$, say. Since $(B + E)/E \cong B/(B \cap E)$, $U/E$ is $RG$%
--isomorphic to some simple $RG$--factor of $B$ and then $G \ne C_G(U/E)$.
But $(B + E)/E \leq A/E\cong\zeta_{RG}^\infty(A)$ and then $G = C_G(U/E)$, a
contradiction that shows $B \leq E$. Hence $E$ contains the $RG$%
--hypereccentric $RG$--submodules of $A$.

\begin{lemma}
\label{l2} Let $G$ be a finite nilpotent group and let $A$ be a $\mathbb{Z}G$%
--module. Suppose that the additive group of $A$ is periodic. Then $A$ has
the $Z$--decomposition.
\end{lemma}

\begin{proof} Since $G$ is finite, $A$ has a local family $\mathcal{L}$ consisting of finite $\mathbb{Z}G$--submodules. If $B\in\mathcal{L}$, applying the results of \cite{ZD1979}, $B$ has the $Z$--decomposition. Pick now $C\in\mathcal{L}$ such that $B\leq C$. Then we have
\[
B = \zeta_{\mathbb{Z}G}^\infty(B)\bigoplus E_{ZG}^\infty(B), C = \zeta_{\mathbb{Z}G}^\infty(C)\bigoplus E_{\mathbb{Z}G}^\infty(C).
\]
Clearly $\zeta_{\mathbb{Z}G}^\infty(B)\leq\zeta_{\mathbb{Z}G}^\infty(C)$ and, since $E_{\mathbb{Z}G}^\infty(C)$ includes every $\mathbb{Z}G$--hypereccentric $\mathbb{Z}G$--submodule, $E_{\mathbb{Z}G}^\infty(B)\leq E_{\mathbb{Z}G}^\infty(C)$. It follows that
\[
\zeta_{\mathbb{Z}G}^\infty(A) = \bigcup_{B\in\mathcal{L}}\zeta_{\mathbb{Z}G}^\infty(B), E_{\mathbb{Z}G}^\infty(A) =
\bigcup_{B\in\mathcal{L}}E_{ZG}^\infty(B).
\]
Therefore $A = \zeta_{\mathbb{Z}G}^\infty(A)\bigoplus E_{\mathbb{Z}G}^\infty(A)$.\end{proof}

\begin{lemma}
\label{l3} Let $G$ be a finite group and $Z$ a $G$--invariant subgroup of
the hypercenter of $G$. Put $|G/Z| = t$. Then there exists a function $f_1$
such that the nilpotent residual of $G$ has order at most $f_1(t)$.
\end{lemma}

\begin{proof} The subgroup $Z$ has a series of $G$--invariant subgroups
\[
\cyc{1} = Z_0\leq Z_1\leq\cdots\leq Z_n\leq Z_{n+1} = Z
\]
whose factors $Z_{j+1}/Z_j$ are $G$--central. Applying a result due to L. A. Kaloujnine \cite{KL1953}, the factor-group $G/C_G(Z)$ is nilpotent. Put $C = C_G(Z)$ so that $Z\leq C_G(C)$. In particular, $|G/C_G(C)|\leq t$. Clearly $C\cap Z\leq\zeta(C)$ and so $C/(Z\cap C)\cong CZ/Z$ is a finite group of order at most $t$. By Wiegold's theorem \cite{WJ1956}, the derived subgroup $D = [C,C]$ has order at most $w(t)$. We note that $D$ is $G$--invariant and $C/D$ is abelian. By the facts proved above, the factor-group $(G/D)/C_{G/D}(C/D)$ is nilpotent. By Lemma \ref{l2}, the $\mathbb{Z}G$--module $C/D$ has the
$Z$--decomposition, that is $C/D = \zeta_{\mathbb{Z}G}^\infty(C/D)\bigoplus E_{RG}^\infty(C/D)$. Clearly, $(C\cap Z)D/D\leq\zeta_{\mathbb{Z}G}^\infty(C/D)$ and then $L/D = E_{\mathbb{Z}G}^\infty(C/D)$ has order at most $t$. Hence $(C/D)/(L/D)$ is $\mathbb{Z}G$--hypercentral. In other words, the hypercenter of $G/L$ contains $C/L$. Since $G/C$ is nilpotent so is $G/L$.
Finally, $|L| = |D| |L/D|\leq tw(t)= tt^m=t^{m+1}$, where $m = \frac{1}{2}(log_p t-1)$ and $p$ is the least prime divisor of $t$, so that $m + 1 =
\frac{1}{2}(log_p t-1) + 1 = \frac{1}{2}(log_p t + 1)$. Therefore, it suffices to put $f_1(t) = t^k$, where $k = \frac{1}{2}(log_p t + 1)$ and $p$ is the least prime divisor of $t$.\end{proof}

\medskip

If $G$ is a group, we denote by $Tor(G)$ the maximal periodic normal
subgroup of $G$. $Tor(G)$ is a characteristic subgroup of $G$ and, if $G$ is
locally nilpotent, $G/Tor(G)$ is torsion-free.

\begin{lemma}
\label{l4} Let $G$ be a finitely generated group and $Z$ a $G$--invariant
subgroup of the hypercenter of $G$. Suppose that $|G/Z| = t$ is finite. Then 
$G$ has a finite normal subgroup $L$ such that $G/L$ is nilpotent. Moreover, 
$|L|\leq f_1(t)$.
\end{lemma}

\begin{proof} Since $G/Z$ is finite, $Z$ is finitely generated. It follows that $Z$ is nilpotent. Moreover, $zl(G)$ is finite. By Baer's theorem \cite{BR1952}, $G$ has a finite normal subgroup $F$ such that $G/F$ is nilpotent.  Being finitely generated, $G/F$ has a finite periodic part $Tor(G/F) = K/F$. As we remarked above, the factor-group $(G/F)/(K/F)\cong G/K = B$  is torsion-free and nilpotent. We have that the subgroup $Z$ is nilpotent and $T = Tor(G)$ is finite. Then $Z$ has a torsion-free normal subgroup $U$ such that the orders of the elements of $Z/U$ are the divisors of some positive integer $k$ (see \cite[Proposition 2]{HK1995} for example). Put $V = Z^k$ so that $V\leq U$ and $V$ is also torsion-free. By construction, $V$ is $G$--invariant and $G/V$ is periodic. Being finitely generated nilpotent--by--finite, $C = G/V$ is finite. By Lemma \ref{l3}, the nilpotent residual $D$ of $C$ has order at most $f_1(t)$.

Clearly $V\cap K = \cyc{1}$. Applying Remak's theorem, we obtain an embedding $G\leq G/V\times G/K = C\times B = H$. Since $B$ is torsion-free nilpotent, the nilpotent residual of $H$ is exactly $D$. It follows that $G/(G\cap D)\cong GD/D\leq H/D$ is nilpotent. This shows that $G\cap D$ includes the nilpotent residual $L$ of $G$. In particular, $L$ is finite and moreover $|L|\leq |G\leq D|\leq |D|\leq f_1(t)$.\end{proof}

\medskip

We are now in a position to show the main result of this paper

\medskip

\noindent\textit{Proof of Theorem B}. Since $G/Z$ is finite, there exists a
finitely generated subgroup $K$ such that $G = KZ$. We pick the family $%
\Sigma$ of all finitely generated subgroups of $G$ that contains $K$.
Clearly $G$ is $FC$--hypercentral and then every finitely generated subgroup
of $G$ is nilpotent--by--finite (see \cite[Proposition 3.19]{KOS2007} for
example). If $U\in\Sigma$, then the hypercenter of $U$ includes a $U$%
--invariant subgroup $U\cap Z = Z_U$ such that $U/Z_U$ is nilpotent and has
order at most $t$. By Lemma \ref{l4}, $U$ has a finite normal subgroup $H_U$
such that $U/H_U$ is nilpotent and $|H_U|\leq f_1(t)$. Being
finite-by--nilpotent, the nilpotent residual $L_U$ of $U$ is finite and $L_U$
has order at most $f_1(t)$.

Pick $Y\in\Sigma$ such that $|L_Y|$ is maximal and let $\Sigma_1$ be the
family of all finitely generated subgroups of $G$ that contains $Y$. Pick $%
U\in\Sigma_1$. Then $Y\leq U$. The factor-group $U/L_U$ is nilpotent and,
since $Y/(Y\cap L_U)\cong YL_U/L_U\leq U/L_U$, $Y/(Y\cap L_U)$ is nilpotent.
It follows that $L_Y\leq Y\cap L_U$ and then $L_Y\leq L_U$. But $|L_Y|$ is
maximal, so that $L_U = L_Y$. In particular, $L_Y$ is normal in $U$ for
every $U\in\Sigma_1$. Then $L_Y$ is normal in $\bigcup_{U\in\Sigma_1}U = G$
and $U/L_Y$ is nilpotent. Thus $G/L_Y$ has a local family of nilpotent
subgroups, that is $G/L_Y$ is locally nilpotent. Then $(G/L_Y)/(ZL_Y/L_Y)$
is nilpotent since it is finite. It follows that $G/L_Y$ is hypercentral,
because the upper hypercenter of $G/L_Y$ includes $ZL_Y/L_Y$.\hfill $\Box$

\end{document}